\documentclass[11pt,leqno]{article}
\usepackage{ graphicx, amsfonts, amsthm, amsxtra, verbatim, multicol}
\usepackage[mathscr]{euscript}
\makeindex
\begin{document}

\begin{center}
{\LARGE\bf Quasi-modular forms attached to elliptic curves: Hecke operators
}
\\
\vspace{.25in} {\large {\sc Hossein Movasati}} \\
Instituto de Matem\'atica Pura e Aplicada, IMPA, \\
Estrada Dona Castorina, 110,\\
22460-320, Rio de Janeiro, RJ, Brazil, \\
E-mail:
{\tt hossein@impa.br} \\
{\tt www.impa.br/$\sim$hossein}
\end{center}

\newtheorem{theo}{Theorem}
\newtheorem{exam}{Example}
\newtheorem{coro}{Corollary}
\newtheorem{defi}{Definition}
\newtheorem{prob}{Problem}
\newtheorem{lemm}{Lemma}
\newtheorem{prop}{Proposition}
\newtheorem{rem}{Remark}
\newtheorem{conj}{Conjecture}
\newtheorem{exe}{Exercise}
\newcommand\diff[1]{\frac{d #1}{dz}} 
\def\End{{\rm End}}              
\def\hol{{\rm Hol}}
\def\sing{{\rm Sing}}            
\def\spec{{\rm Spec}}            
\def\cha{{\rm char}}             
\def\Gal{{\rm Gal}}              
\def\jacob{{\rm jacob}}          
\def\Z{\mathbb{Z}}                   
\def\O{{\cal O}}                       

\def\C{\mathbb{C}}                   
\def\as{\mathbb{U}}                  
\def\ring{{\sf R}}                             
\def\R{\mathbb{R}}                   
\def\N{\mathbb{N}}                   
\def\A{\mathbb{A}}                   

\def\D{\mathbb{D}}                   
\def\uhp{{\mathbb H}}                
\newcommand\ep[1]{e^{\frac{2\pi i}{#1}}}
\newcommand\HH[2]{H^{#2}(#1)}        
\def\Mat{{\rm Mat}}              
\newcommand{\mat}[4]{
     \begin{pmatrix}
            #1 & #2 \\
            #3 & #4
       \end{pmatrix}
    }                                
\newcommand{\matt}[2]{
     \begin{pmatrix}                 
            #1   \\
            #2
       \end{pmatrix}
    }
\def\ker{{\rm ker}}              
\def\cl{{\rm cl}}                
\def\dR{{\rm dR}}                

\def\hc{{\mathsf H}}                 
\def\Hb{{\cal H}}                    
\def\GL{{\rm GL}}                
\def\pedo{{\cal P}}                  
\def\PP{\tilde{\cal P}}              
\def\cm {{\cal C}}                   
\def\K{{\mathbb K}}                  
\def\k{{\mathsf k}}                  
\def\F{{\cal F}}                     
\def\M{{\cal M}}
\def\RR{{\cal R}}
\newcommand\Hi[1]{\mathbb{P}^{#1}_\infty}
\def\pt{\mathbb{C}[t]}               
\def\W{{\cal W}}                     
\def\Af{{\cal A}}                    
\def\gr{{\rm Gr}}                
\def\Im{{\rm Im}}                
\newcommand\SL[2]{{\rm SL}(#1, #2)}    
\newcommand\PSL[2]{{\rm PSL}(#1, #2)}  
\def\Res{{\rm Res}}              

\def\L{{\cal L}}                     
\def\Aut{{\rm Aut}}              
\def\any{R}                          
\newcommand\ovl[1]{\overline{#1}}    

\def\per{{\sf  pm}}  
\def\T{{\cal T }}                    
\def\tr{{\sf tr}}                 
\newcommand\mf[2]{{M}^{#1}_{#2}}     
\newcommand\bn[2]{\binom{#1}{#2}}    
\def\ja{{\rm j}}                 
\def\Sc{\mathsf{S}}                  
\newcommand\es[1]{g_{#1}}            
\newcommand\V{{\mathsf V}}           
\newcommand\Ss{{\cal O}}             
\def\rank{{\rm rank}}                
\def\diag{{\rm diag}}
\def\BM{{\sf H}}
\def\Fi{{S}}
\def\Ra{{\rm R}}

\def\Q{{\mathbb Q}}                   

\def\P{\mathbb{P}}

\begin{abstract}
In this article we introduce Hecke operators on the differential algebra of geometric 
quasi-modular forms. As an application for each natural number $d$ we construct  a vector field in six dimensions which determines uniquely 
the polynomial relations between the Eisenstein series of weight $2,4$ and $6$ and their transformation under multiplication of the argument by $d$, 
and in particular, 
it determines uniquely the modular curve of degree $d$ isogenies between elliptic curves.
\end{abstract}

\section{Introduction}
The theory of quasi-modular forms was first introduced by Kaneko and Zagier in \cite{kaza95} due to its applications in mathematical physics.
In \cite{ho06-2} and \cite{ho14} we have described how one can introduce quasi-modular forms in the framework of the algebraic geometry of elliptic
curves, and in particular, how the Ramanujan differential equation between Eisenstein series can be derived from the Gauss-Manin connection of 
families of elliptic curves. In the present article we proceed further our investigation and  we introduce Hecke operators for quasi-modular forms 
and give some applications. However, the main motivation behind this work is to prepare the ground for similar topics in the case of Calabi-Yau varieties, 
see \cite{ho11}.

In  \cite{CDLW}  the authors describe a differential equation in the $j$-invariant of two elliptic curves which is 
tangent to all modular curves  of degree $d$ isogenies of elliptic curves. This differential equation can be derived from 
the Shwarzian differential equation of the $j$-function and the later can be calculated from
the Ramanujan differential equation between Eisenstein series. 
This suggests that there must be a relation between Ramanujan differential equation and modular curves. 
In this article we also establish  this relation.

Consider the Ramanujan ordinary differential equation 
\begin{equation}
\label{raman}
{\rm R}:
 \left \{ \begin{array}{l}
\dot s_1=\frac{1}{12}(s_1^2-s_2) \\
\dot s_2=\frac{1}{3}(s_1s_2-s_3) \\
\dot s_3=\frac{1}{2}(s_1s_3-s_2^2)
\end{array} \right. \dot s_k=\frac{\partial s_k}{\partial \tau}
\end{equation} 
which is satisfied by Eisenstein series:
\begin{equation}
\label{eisenstein}
s_i(\tau)=a_iE_{2i}(q):=a_i\left(1+b_i\sum_{n=1}^\infty \left (\sum_{d\mid n}d^{2i-1}\right )q^{n}\right ),
\end{equation}
$$
i=1,2,3,  \ q=e^ {2\pi i\tau},\ \Im(\tau)>0
$$
and 
$$
(b_1,b_2,b_3)=
(-24, 240, -504),\ \ (a_1,a_2,a_3)=(2\pi i,(2\pi i)^2 ,(2\pi i)^3).
$$
The algebra of modular forms for $\SL 2\Z$ is generated by the Eisenstein series $E_4$ and $E_6$ and 
all modular forms for congruence groups are algebraic  over the field $\C(E_4,E_6)$, see for instance \cite{zag123}. In a similar way
the algebra of quasi-modular forms for $\SL 2\Z$ is generated by $E_2,E_4$ and $E_6$, see for instance \cite{maro05, ho14}, and we have:
\begin{theo}\rm
\label{7nov2011}
For $i=1,2,3$ and $d\in\N$, there is a unique homogeneous polynomial $I_{d,i}$ of degree $i\cdot \psi(d)$, 
where  $\psi(d):=d\prod_{p}(1+\frac{1}{p})$ is the Dedekind 
$\psi$ function and $p$ runs through primes $p$ dividing $d$,  
in the weighted ring
\begin{equation}
\label{28mar2012}
\Q[t_i,s],\ {\rm weight}(t_i)=i,\ {\rm weight}(s_j)=j,\ j=1,2,3
\end{equation}
and monic in the variable $t_i$ such that $t_i(\tau):=d^{2i}\cdot s_i(d\cdot \tau),s_1(\tau),s_2(\tau), s_3(\tau)$ satisfy the algebraic relation:
$$
I_{d,i}(t_i,s_1,s_2, s_3)=0.
$$
Moreover, for $i=2,3$ the polynomial $I_{d,i}$ does not depend on $s_1$.
\end{theo}
The novelty of Theorem \ref{7nov2011} is mainly due to the case $i=1$.  Since the Eisenstein series $E_2,E_4$ and $E_6$ are algebraically independent over $\C$, we can consider 
$s_i,t_i,\ i=1,2,3$ as indeterminate variables. 
In this way we regard $(t,s)=(t_1,t_2,t_3,s_1,s_2,s_3)$ as coordinates of the affine variety 
$\A_\k^6$, where $\k$ is any field of characteristic zero and not necessarily algebraically closed.
The Ramanujan's ordinary differential equation (\ref{raman}) is considered as a vector field in $\A_\k^3$ with the coordinates $s$.
It can be shown  that the curve given by $I_{d,2}=I_{d,3}=0$ in the weighted projective space 
$\P ^{(2,3,2,3)}_\C$ with the coordinates $(t_2,t_3,s_2,s_3)$  is biholomorphic to the modular curve
$$
X_0(d):=\Gamma_0(d)\backslash (\uhp\cup\Q),\ \hbox{   where}\ \ \ \Gamma_0(d):=\{\mat{a}{b}{c}{d}\in \SL 2\Z\mid c\equiv_d0\}.
$$
The Shimura-Taniyama conjecture, 
now the modularity theorem, states that any elliptic  over $\Q$ must appear in the decomposition of the Jacobian of $X_0(d)$, where $d$ is the conductor of the elliptic curve 
(see \cite{wil95} for the case of semi-stable elliptic curves and \cite{bcdt} for the case of all elliptic curves). 
Computing explicit equations for $X_0(d)$ in terms of the variables $j_1=1728\frac{t_2^3}{t_2^3-t_3^2}$ and $j_2=1728\frac{s_2^3}{s_2^3-s_3^2}$
has many applications in number theory and it has been done by many authors, see for instance \cite{yui78} and the references therein.

Let $\Ra_t$, respectively $\Ra_s$, be the 
Ramanujan vector field in $\A_\k^3$ with $t$ coordinates, respectively $s$ coordinates. In $\A_\k^6=\A^3_\k\times \A^3_\k$ with the coordinates system  $(t,s)$ 
we consider the vector field:  
$$
\Ra_d:= R_t+d\cdot R_s.
$$
Let $V_d$ be the affine subvariety of $\A^6_\k$ given by $I_{d,1}=I_{d,2}=I_{d,3}=0$. 
\begin{theo}
\label{23nov2011}
The vector field $\Ra_d$ is tangent to the  affine variety $V_d$.
\end{theo}
I do not know the complete classification of all $\Ra_d$-invariant algebraic subvarieties of $\A_\k^ 6$. 
We consider $\Ra_d$ as a differential operator:
$$
\k[t,s]\to \k[t,s],\ f\mapsto \Ra_d(f):=df(\Ra_d).
$$
From Theorem \ref{23nov2011} and the fact that $V_d$ is irreducible (see \S\ref{27nov2011-7}), it follows that:
$$
R_d^j(I_{d,i})\in {\rm Radical}\langle I_{d,1}, I_{d,2}, I_{d,3}\rangle ,\ \ \   i=1,2,3, \ \ j\in\N\cup\{0\}.
$$
Note that the ideal $\langle I_{d,1}, I_{d,2}, I_{d,3}\rangle \subset \k[t,s]$ may not be radical. 
We can compute $I_{d,i}$'s using the $q$-expansion of Eisenstein series, see \S\ref{27nov2011-7}.  This method works only
for small degrees $d$. However, for an arbitrary $d$ we can introduce some elements in the radical of the ideal generated by $I_{d,i},\ i=1,2,3$. 
Let
$$
J_{d,i}=\det
\begin{pmatrix}
 \alpha_1&\alpha_2&\cdots&\alpha_{m_{d,i}}\\
 \Ra_d(\alpha_1)&\Ra_d(\alpha_2)&\cdots&\Ra_d(\alpha_{m_{d,i}})\\
 \vdots &\vdots&\cdots&\vdots\\
 \Ra_d^{m_{d,i}-1}(\alpha_1)&\Ra_d^{m_{d,i}-1}(\alpha_2)&\cdots&\Ra_d^{m_{d,i}-1}(\alpha_{m_{d,i}})\\
\end{pmatrix},
$$
where for $i=1$, $\alpha_j$'s are the monomials: 
\begin{equation}
\label{8july2011}
t_i^{a_0}s_1^{a_1}s_2^{a_2}s_3^{a_3},\ i\cdot \psi(d)=ia_0+a_1+2a_2+3a_3,\ a_0,a_1,a_2,a_3\in\N_0
\end{equation}
and for $i=2,3$, $\alpha_j$'s are the above monomials with $a_1=0$. 
The polynomial $J_{d,i}$   is weighted homogeneous of degree 
$$
i\cdot \psi(d)+i\cdot \psi(d)+1+i\cdot\psi(d)+2+\cdots+i\cdot\psi(d)+m_{d,i}-1=m_{d,i}\cdot i\cdot \psi(d)+\frac{m_{d,i}\cdot (m_{d,i}-1)}{2}.
$$

\begin{theo}
\label{07nov2011}
We have 

$$
J_{d,i}\in  {\rm Radical}\langle I_{d,1}, I_{d,2}, I_{d,3}\rangle  ,\ i=1,2,3.
$$
and so
$$
J_{d,i}(d^{2i}\cdot s_i(d\cdot \tau),s_1(\tau),s_2(\tau),s_3(\tau))=0,\ i=1,2,3.
$$
\end{theo}
Throughout the text we will state our results over a  field $\k$ of characteristic zero and not necessarily algebraically closed. Such results
are valid if and only if the same results are valid over the algebraic closure $\bar \k$ of $\k$. 
By Lefschetz principle, see for instance \cite{si86} p.164,
it is enough to prove such results over the complex numbers.

The article is organized in the following way. In \S  \ref{27nov2011-1} and \S \ref{joacir?} we recall the definition of full 
quasi-modular forms in the framework of both algebraic geometry and complex analysis. In \S \ref{27nov2011-2} we describe some facts relating isogenies and algebraic de Rham
cohomology of elliptic curves. Using isogeny of elliptic curves we introduce Hecke operators in \S \ref{27nov2011-3}.
 Theorem \ref{7nov2011}, Theorem \ref{23nov2011} and
Theorem \ref{07nov2011} are respectively proved  in  \S\ref{27nov2011-4}, \S\ref{27nov2011-5} and \S\ref{27nov2011-6}. 
Finally, in \S\ref{27nov2011-7} we give some 
examples. 

The main idea behind the proof of Theorem \ref{07nov2011} is due to J. V. Pereira in \cite{pereira}. Here, I would like to thank him for teaching me such
an elegant and simple argument.

\section{Geometric quasi-modular forms}
\label{27nov2011-1}
In this section we recall  some definitions and theorems in  \cite{maro05, ho06-2}. The reader is also referred to \cite{ho14} 
for a complete account 
of quasi-modular forms in a geometric context. Note that in \cite{ho14} the $t$ parameter is in fact 
$(\frac{1}{12}t_1, 12\frac{1}{12^2}t_2,8\frac{1}{12^3}t_3)$.
Let $\k$ be any field of characteristic zero and let 
$E$ be an elliptic curve over $\k$. The first algebraic de Rham cohomology of $E$, namely $H^1_\dR(E)$, is a $\k$-vector
space of dimension two and it has a one dimensional space $F^1$ consisting of elements represented by regular differential 1-forms on $E$. 
\begin{theo}
\label{theo-1}(\cite{ho14}, \S 5.5)
The set $T(\k)=\A_\k^3\backslash\Delta$, where $\Delta:=\{(t_1,t_2,t_3)\in \A_\k^ 3\mid t_2^3-t_3^2=0\}$, 
is the moduli of the pairs $(E,\omega)$, where 
$E$ is an elliptic curve and $\omega\in H_\dR^1(E)\backslash F^1$. For $(t_1,t_2,t_3)\in T(\k)$, the corresponding pair $(E,\omega)$
is given by
$$
E: 3y^2=(x-t_1)^3-3t_2(x-t_1)-2t_3,\   \ \omega=\frac{1}{12}\frac{xdx}{y}.
$$
\end{theo}
From now on an element of $T(\k)$ is denoted either by $(t_1,t_2,t_3)$ or $(E,\omega)$. We can regard  $t_i$ as a rule which for any pair
$(E,\omega)$ as above  it associates an element $t_i=t_i(E,\omega)\in\k$. We will also use $t_i$ as an indeterminate variable or an element in $\k$, 
being clear from the
text which we mean. A full quasi-modular form $f$ of weight $m$ and differential order $n$ 
is a homogeneous 
element  in  the $\k$-algebra
$$
\mf{}{}:=\k[t_1,t_2,t_3],\ {\rm weight}(t_i)=2i,\ i=1,2,3,
$$
with $\deg(f)=m$ and $\deg_{t_1}f\leq n$. The set of such quasi-modular forms is denoted by $\mf{n}{m}$.

For a pair $(E,\omega)\in T(\k)$ we have also a unique element $\alpha\in F^1$ satisfying
$\langle \alpha,\omega\rangle=1$, where $\langle \cdot,\cdot\rangle$ is the intersection form in the de Rham cohomology, see for instance
\cite{ho14} \S2.10.  For this reason we sometimes use $(E,\{\alpha,\omega\})$ instead of $(E,\omega)$. 
 The algebraic groups ${\mathbb G}_a:=(\k,+)$ and 
${\mathbb G}_m:=(\k-\{0\}, \cdot)$ act from the right on $T(\k)$:
\begin{eqnarray*}
(E,\omega)\bullet k&:=&(E,k\omega),\ k\in {\mathbb G}_m,\\
(E,\omega)\bullet k&:=& (E,\omega+k\alpha),\ k\in {\mathbb G}_a 
\end{eqnarray*}
and so they act from the left on $\mf{}{}$. It can be shown that $\mf{n}{m}$ is invariant under these actions and the functions 
$t_i:T\to\k,\ i=1,2,3$ satisfy 
$$
k\bullet t_1=t_1+k,\ k\bullet t_i=t_i,\  i=2,3\ \ \ k\in {\mathbb G}_a,
$$
$$
k\bullet t_i=k^{-2i}t_i,\  i=1,2,3\ \ \ k\in {\mathbb G}_m.
$$
Let $\Ra$ be the the Ramanujan vector field in $T$. 
It is the unique vector field in $T$ which
satisfies $\nabla_{\Ra}\alpha=-\omega,\ \nabla_\Ra\omega=0$, where $\nabla$ is the Gauss-Manin connection of the universal family 
of elliptic curves over $T$,
see for instance \cite{ho14} \S2. 
The $\k$-algebra of full quasi-modular forms has a differential structure which is given 
by:
$$
\mf{n}{m}\to \mf{n+1}{m+2},\ t\mapsto \Ra(t):=\sum_{i=1}^3\frac{\partial t}{\partial t_i}
{\Ra}_i,
$$
where $\Ra=\sum_{i=1}^3 \Ra_i\frac{\partial}{\partial t_i}$ is the Ramanujan vector field.

\section{Holomorphic quasi-modular forms}
\label{joacir?}
Now, let us assume that $\k=\C$. The period domain is defined to be
\begin{equation}
\label{perdomain}
\pedo:= \left \{
\mat {x_1}{x_2}{x_3}{x_4}\mid x_i\in\C,\ x_1x_4-x_2x_3=1,\  \Im(x_1\ovl{x_3})>0 \right \}.
\end{equation}
We let the group $\SL 2\Z$ act from the left  on $\pedo$ by usual multiplication of matrices. In $\pedo$ we consider the vector field
\begin{equation}
 \label{9jan2012}
X=-x_2\frac{\partial}{\partial x_1}-x_4\frac{\partial}{\partial x_3}
\end{equation}
which is invariant under the action of $\SL 2\Z$ and so it 
induces a vector field in the complex manifold $\SL 2\Z\backslash \pedo$. For simplicity we denote it again by $X$.
The Poincar\'e upper half plane $\uhp$ is embedded in $\pedo$ in the following
way:
$$
\tau\rightarrow \mat{\tau}{-1}{1}{0}
$$
and so we have a canonical map $\uhp\to \SL 2\Z\backslash \pedo$.
\begin{theo}
\label{theo-2}
(\cite{ho14} \S8.4 and \S8.8)
 The period map
$$
\per: T(\C)\to \SL 2\Z\backslash \pedo
$$
$$
t\mapsto
\left [\frac{1}{\sqrt{-2\pi i}}\mat
{\int_{\delta} \alpha}
{\int_{\delta}\omega }
{\int_{\gamma} \alpha}
{\int_{\gamma}\omega } \right ]
$$
is a biholomorphism, where $\{\delta,\gamma\}$ is a basis of $H_1(E,\Z)$ with $\langle\delta,\gamma\rangle=-1$. Under this biholomorphism
the Ramanujan vector field is mapped to $X$. The pull-back of $t_i$ by the composition

\begin{equation}
\label{khoshgel?}
\uhp\to \SL 2\Z\backslash \pedo\stackrel{\per^{-1}}{\to} T(\C)\hookrightarrow\A_\C^3,
\end{equation}
 is the  Eisenstein series $a_iE_{2i}(e^{2\pi i \tau})$ in (\ref{eisenstein}).
\end{theo}
The algebra of full holomorphic quasi-modular forms  is the pull-back of $\k[t_1,t_2,t_3]$ under the composition  
(\ref{khoshgel?}). We can also introduce it in a classical way using functional equations plus growth conditions:
a holomorphic function $f$ on $\uhp$ is called a (holomorphic) quasi-modular form of weight $m$ and differential order $n$ 
if the following two conditions are satisfied:
\begin{enumerate}
\item
There are holomorphic functions $f_i,\ i=0,1,\ldots,n$ on $\uhp$ such that 
\begin{equation}
\label{4feb05}
(cz+d)^{-m}f(Az)=\sum_{i=0}^{n}\bn{n}{i} c^i(cz+d)^{-i}f_i, \ \forall A=\mat{a}{b}{c}{d}\in \SL 2\Z.
\end{equation}
\item
$f_i, i=0,1,2,\ldots,n$ have finite growths when $\Im(z)$ tends to $+\infty$, i.e.
$$
\lim_{\Im(z)\to +\infty}f_i(z)=a_{i,\infty} <\infty,\ a_{i,\infty}\in\C.
$$
\end{enumerate}
For the proof of the equivalence of both notions of quasi-modular forms see \cite{ho14} \S 8.11.
We have $f_0=f$ and the associated functions $f_i$ are unique. In fact, $f_i$ is a quasi-modular form of weight $m-2i$ and differential order $n-i$ and
with associated functions $f_{ij}:=f_{i+j}$. 
It is useful to define
\begin{equation}
\label{26feb05} f||_mA:=(\det A)^{m-1}\sum_{i=0}^{n}\bn{n}{i}
(\frac{-c}{\det(A)})^i(cz+d)^{i-m}f_{i}(Az),\ 
\end{equation}
$$
A=\mat{a}{b}{c}{d}\in \GL (2, \R),\ f\in\mf nm.
$$
In this way, the
equality (\ref{4feb05}) is written in the form
\begin{equation}
\label{13feb05} f=f||_{m}A, \forall A\in\SL 2\Z
\end{equation}
and we have
$$
f||_m A=f||_m(BA),\ \forall A\in \GL(2,\R),\ B\in\SL 2\Z,\ f\in\mf{n}{m}.
$$
It can be proved that the algebra of full quasi-modular forms is generated by the Eisenstein series $E_{2i},\ i=1,2,3$. 
For further details on holomorphic quasi-modular forms see \cite{maro05, ho06-2, ho14}.
\section{Isogeny of elliptic curves}
\label{27nov2011-2}
Let $(E_1,0_1)$ and $(E_2,0_2)$ be two elliptic curves over the field $\k$. Here, $0_i\in E_i(\k),\ i=1,2$ is the neutral element of the group $E_i(\k)$. We say that $E_1$ is isogenous to $E_2$ over $\k$ if
there is a non-constant morphism of algebraic curves over $\k$ $f:E_1\to E_2$ which sends $0_1$ to $0_2$. It can be shown that 
$f$ induces a morphism of groups $E_1(\k)\to E_2(\k)$.  We also say that $f$ is the isogeny between $E_1$ and $E_2$ over $\k$.
For all points $p\in E(\bar \k)$ except a finite number, $\#f^{-1}(p)$ is a fixed
number which we denote it by $\deg(f)$. Here, we have considered $f$ as a map from $E_1(\bar \k)$ to
$E_2(\bar\k)$. Since $f$ is a morphism of groups, for a point $q\in f^{-1}(p)$ the map $x\mapsto x+q$ induces a 
bijection $f^{-1}(0_1)\cong f^{-1}(p)$. We conclude that for all points 
$p\in E_2(\bar \k)$, the set
$f^{-1}(p)$ has $\deg(f)$ points (and hence $f$ has no ramification points).
\begin{prop}
\label{noite19may}
 Let $f:E_1\to E_2$ be an isogeny of degree $d$. Then for all $\omega,\alpha\in\ H^1_\dR(E_2)$ we have
$$
\langle f^*\omega,f^*\alpha \rangle=d\cdot \langle \omega,\alpha \rangle.
$$
\end{prop}
Here, $\langle \cdot,\cdot\rangle:H_\dR^1(E)\times H_\dR^1(E)\to \k$ is the intersection form in the de Rham cohomology, see \cite{ho14} \S2.10.
\begin{proof}
 It is enough to prove the proposition over algebraically closed field. Since the above formula
is $\k$-linear in both $\omega$ and $\alpha$, it is enough to prove it in the case 
$\omega=\frac{dx}{y},\ \alpha=\frac{xdx}{y}$, where $x,y$ are the Weierstrass coordinates of 
$E$. Since $\langle \frac{dx}{y},\frac{xdx}{y}\rangle=1 $,
we have  to prove that $\langle f^*(\frac{dx}{y}), f^*(\frac{xdx}{y})\rangle=d$. 
Let $f^{-1}(0_2)=\{p_1,p_2,\ldots, p_d\}$. 
The differential form $f^*(\frac{dx}{y})$ is again a regular
differential form and $f^*(\frac{xdx}{y})$ has poles of order two at each $p_i$. 
Consider the covering $U=\{U_0,U_1\}$ of $E_1$, where $U_0=E_1\backslash f^{-1}(0_E)$ and $U_1$ is any other open
set which contains $f^{-1}(0_E)$. The differential forms $f^*(\frac{x^idx}{y}), \ i=0,1$ as elements in $H^1_\dR(E_1)$ are represented by
the pairs
$$
(\frac{d\tilde x}{\tilde y}, \frac{d\tilde x}{\tilde y}) , \ \ (\frac{\tilde xd\tilde x}{\tilde y}, \frac{\tilde x d\tilde x}{\tilde y}-\frac{1}{2}d(\frac{\tilde y}{\tilde x})),
$$
where $\tilde x=f^*x,\ \tilde y=f^*y$. 
We have $\frac{d\tilde x}{\tilde y}\cup \frac{\tilde xd\tilde x}{\tilde y}=\{\omega_{01}\}$, where $\omega_{01}=\frac{-1}{2}
\frac{d\tilde x}{\tilde x}$ and so
$$
\langle f^*(\frac{dx}{y}), f^*(\frac{xdx}{y})\rangle=
\langle \frac{d\tilde x}{\tilde  y}, \frac{\tilde xd\tilde x}{\tilde y}\rangle=
\sum_{i=1}^d {\rm Residue}(\frac{-1}{2}
\frac{d\tilde x }{\tilde x}, p_i)=\sum_{i=1}^d 1=d.
$$
\end{proof}

\begin{prop}
\label{29feb2012}
We have:
\begin{enumerate}
\item
Let $f: E_1\to E_2$ be an isogeny defined over $\k$. The induced map $f^*:H_\dR^1(E_2)\to H_\dR^1(E_1)$ is an isomorphism.
\item
Let $[d]_E: E\to E$ be the multiplication by $d\in\N$ map. We have $[d]^*_E: H_\dR^1(E)\to H_\dR^1(E),\ \omega\mapsto d\cdot \omega$.
\end{enumerate} 
\end{prop}
\begin{proof}
In the complex context,  $E=\C/\langle \tau,1\rangle$ and $[d]_E$ is induced by $\C\to \C,\ z\mapsto d\cdot z$. Moreover,  
a basis of the $C^\infty$ 
de Rham  cohomology is given by $dz,\ d\bar z$. This proves the second part of the proposition.
For the first part we take the dual isogeny and use the first part.
\end{proof}

For $E$ an elliptic curve over an algebraically closed field $\k$ of characteristic zero, the number of isogenies $f: E_1\to E$ of degree $d$
is equal to $\sigma_{1}(d):=\sum_{c\mid d}c$. To prove this we may work in the complex context and assume that 
$E=\C/\langle \tau,1\rangle$. The number of such isogenies is the number of subgroups of order $d$ of $(\Z/d\Z)^2$.
\section{Geometric Hecke operators}
\label{27nov2011-3}
In this section all the algebraic objects are defined over $\k$ unless it is mentioned explicitly.
Let $d$ be a positive integer. The Hecke operator $T_d$ acts on the space of full 
quasi-modular forms as follow:
$$
T_d: \mf{n}{m}\to \mf nm,\ \
$$ 
$$
 T_d(t)(E,\omega)=\frac{1}{d}\sum_{f:E_1\to E,\  \deg(f)=d}  t(E_1, f^*\omega), \  \ t\in \mf nm
$$
where the sum runs through all isogenies $f:E_1\to E$ of degree $d$ defined over 
$\bar\k$. Since $(E,\omega)$ and $t$ are defined over $\k$, $T_d(t)(E,\omega)$ is invariant under $\Gal (\bar\k/\k)$ and so it is in the field $\k$. This 
implies that $T_d(t)$ is defined over $\k$.
The statement $T_d\in \k[t_1,t_2,t_3]$ is not at all clear. 
In order to prove this,   we assume that $\k=\C$ and we prove the same statement for holomorphic quasi-modular forms, see \S \ref{1march2012}.
The functional equation of $T_dt$ with respect to the action of the algebraic groups  ${\mathbb G}_m$ and  ${\mathbb G}_a$ can be proved in the algebraic context as follow:
\begin{prop}
\label{4nov2011}
The action of ${\mathbb G}_m$ commutes with Hecke operators, that is,
\begin{equation}
\label{CFE}
k\bullet T_d(t)=T_d(k\bullet t),\ t\in \mf{}{},\ k\in {\mathbb G}_m
\end{equation}
and the action of ${\mathbb G}_a$ satisfies:
$$
k\bullet T_d(t)=T_d((d \cdot k)\bullet t),\ t\in \mf{}{},\ k\in {\mathbb G}_a.
$$
\end{prop}
\begin{proof}
The first equality is trivial:
\begin{eqnarray*}
(k\bullet T_d(t)) (E,\omega) &=& T_d(t)(E, k\omega)=\frac{1}{d}\sum t(E_1, f^*(k\omega))\\
&=&
\frac{1}{d}\sum (k\bullet t)(E_1, f^*(\omega))=T_d(k\bullet t)(E,\omega)
\end{eqnarray*}
For the second equality we use 
Proposition \ref{noite19may}:
\begin{eqnarray*}
(k\bullet T_d(t)) (E,\omega) &=& T_d(t)(E, \omega+k\alpha)=\frac{1}{d}\sum t(E_1, f^*(\omega+k\alpha))\\
&=&
\frac{1}{d}\sum t(E_1, f^*(\omega)+ d\cdot k f^*(\frac{1}{d}\alpha))=\frac{1}{d}\sum  ((d\cdot k)\bullet t)(E_1, f^*(\omega))\\
&=&      T_d(d\cdot k\bullet t )(E,\omega).
\end{eqnarray*}
\end{proof}

We can also define the Hecke operators  in the following way:
$$
 T_d(t)(E,\omega)=d^{m-1}\sum_{g:E\to E_1,\  \deg(g)=d}  t(E_1, g_*\omega), 
\ t\in \mf nm
$$
where the sum runs through all isogenies $g:E\to E_1$ of degree $d$ defined over 
$\bar \k$. Both definitions of $T_d(t)$ are equivalent: for an isogeny $f: E_1\to E$ of degree $d$ defined over $\bar \k$ we have a unique dual isogeny 
$g: E\to E_1$ such that 
$$
f\circ g=[d]_E,\ \ 
g\circ f=[d]_{E_1}.
$$ 
Therefore by Proposition \ref{29feb2012}  we have $d\cdot g_*\omega= [d]_{E_1}^*(g_*\omega)=f^*\omega$ and so
$$
t(E_1, f^*\omega)=t(E_1, d \cdot g_*\omega)=d^m t(E_1, g_*\omega).
$$
It can be shown that the geometric Eisenstein modular form $G_k$ (see \cite{ho14} \S 6.5) is an eigenform with eigenvalue 
$$
\sigma_{k-1}(d):=\sum_{c\mid d}c^{k-1}
$$ for the Hecke operator $T_d$, that is 
$$
T_dG_k=\sigma_{k-1}(d)G_k,\ \ d\in\N,\ \ k\in 2\N
$$
see for instance \cite{maro05}. 

The differential operator $\Ra: \mf{}{}\to \mf{}{}$ and the Hecke operator $T_d$ commute, that is
$$
{\Ra}\circ T_d=T_d\circ {\Ra}, \ \forall d\in\N.
$$
For the proof we may assume that $\k=\C$. In this way using Theorem \ref{theo-2} it is enough to prove the 
same statement for holomorphic quasi-modular forms, see for instance \cite{ho06-2} Proposition 4.

%

\section{Holomorphic Hecke operators}
\label{1march2012}
In this section we want to use the biholomorphism in Theorem \ref{theo-2} and describe the Hecke operators on holomorphic 
quasi-modular forms. Let us take $\k=\C$. 
\begin{prop}
 The $d$-th Hecke operator on the vector space of quasi-modular forms of weight $m$ and differential order $n$ is given by
$$
T_d:\mf{n}{m}\to \mf{n}{m},\ \ T_df=\sum_{A} f||_mA,
$$
where $A$ runs through the the set $\SL 2\Z\backslash \Mat_d (2,\Z)$ and  $||$ is the double slash operator (\ref{26feb05}) for quasi-modular forms.
\end{prop}
\begin{proof}
Let us consider two points $(E_i, \{\alpha_i, \omega_i\}),\ \ i=1,2$ in the moduli space $T$. 
Let us also consider a $d$-isogeny $f:E_1\to E_2$ with
$$
f^*\omega_2=\omega_1, \ f^*\alpha_2=d\cdot \alpha_1.
$$ 
We can take a symplectic basis $\delta_1,\gamma_1$ of $H_1(E_1,\Z)$ and  $\delta_2,\gamma_2$ of $H_1(E_2,\Z)$
such that 
$$
f_*[\delta_1,\gamma_1]^\tr=A[\delta_2,\gamma_2]^\tr,
$$
where $A$ is an element in  the set  $\Mat_d(2,\Z)$ of $2\times 2$ matrices with
coefficients in $\Z$ and with determinant $d$. From another side we have
$$
[\alpha_1,\omega_1]B=f^*[\alpha_2,\omega_2],\ \hbox{ where } B=\mat{d}{0}{0}{1}
$$
Therefore, if the period matrix associated to $(E_i,\{\alpha_i, \omega_i\},\{\delta_i\,\gamma_i\}),\ i=1,2$ 
is denoted respectively by by $x'$ and $x$ then 
$$
x'B=Ax.
$$ 
 Using Theorem \ref{theo-2}, the $d$-th Hecke operator acts on the space of $\SL 2\Z$-invariant holomorphic functions on 
$\pedo$ by:
$$
T_dF(x)=\frac{1}{d}\sum_{ A} F(A xB^{-1}),\  x\in\pedo
$$
where $A=\mat{a_1}{b_1}{c_1}{d_1}$ runs through $\SL 2\Z\backslash  \Mat_d (2,\Z)$. 
Let $f\in\mf{n}{m}$ be a holomorphic quasi-modular form defined on the upper half plane. By definition there is a geometric modular form 
$\tilde f\in\C[t_1,t_2,t_3]$  such that $f$ is the pull-back of $\tilde f$ by the composition (\ref{khoshgel?}). Let $F$ be the 
holomorphic function on $\pedo$ obtained by the push-forward of $\tilde f$ by the period map and then its pull-back by 
$\pedo\to \SL 2\Z \backslash \pedo$.
We have
\begin{eqnarray*}
 T_df(\tau) &=& 
T_dF\mat{\tau}{-1}{1}{0} \\
&=& \frac{1}{d}\sum_{A} F(A\mat{\tau}{-1}{1}{0} B^{-1})\\
&=& \frac{1}{d}\sum_{A} F(\mat{A\tau}{-1}{1}{0}\mat{d^{-1}(c_1\tau+d_1)}{-c_1}{0}{(c_1\tau+d_1)^{-1}d})\\
&=&  \frac{1}{d} \sum_{A}  d^{m}(c_1\tau+d_1)^{-m}\sum_{i=1}^n\bn{n}{i}(-c_1)^i(c_1\tau+d_1)^{i}d^{-i}f_i(A(\tau))\\
&=&  \sum_{A} f||_mA
\end{eqnarray*}
\end{proof}
One can take the representatives 
$$
\{A_i\}:=
\left. \left \{\mat{\frac{d}{c}}{b}{0}{c}\right| c\mid d,\ 0\leq b<c \right \}
$$
for the quotient $ \SL 2\Z\backslash  \Mat_d (2,\Z)$ and so
 \begin{equation}
\label{quedia}
T_df(\tau)=\frac{1}{d} \sum_{cc'=d,\ \  0\leq b<c}   
{c'}^{m}f{\Big (}\frac{c'\tau+b}{c}{\Big )}.
\end{equation}
In a similar way to the case of modular forms (see \cite{apo90} \S 6) one can
check that
$$
T_p\circ T_q=\sum_{d\mid (p,q)}d^{m-1}T_{\frac{pq}{d^2}}.
$$
If we write $f=\sum_{n=0}^\infty f_nq^n$ then we have:
$$
(T_df)_n=\sum_{c\mid (d,n)} c^{m-1}f_{\frac{nd}{c^2}}.
$$
In particular if we set $n=0$ then the constant term of $T_d(f)$ is $f_0\sigma_{m-1}(d)$. If $f$ is the eigenvector of $T_d$ 
and the constant term of $f$ is non-zero then the corresponding eigenvalue is $\sigma_{m-1}(d)$.

\section{Refined Hecke operators}
\label{payebabanemikham}
Let $W_d$ be the set of subgroups of $\frac{\Z}{d\Z}\times \frac{\Z}{d\Z}$ of order $d$ and 
$S_d$  be the set (up to isomorphism) of finite groups of order $d$ and generated by at most two elements. We have a canonical surjective map
$W_d\to S_d$.
\begin{prop}
We have bijections
\begin{equation}
\label{cheragazgerefti}
\SL 2\Z\backslash \Mat_d(2,\Z)\cong W_d,
\end{equation}
\begin{equation}
 \label{clarice}
\SL 2\Z\backslash \Mat_d(2,\Z)/\SL 2\Z\cong S_d,
\end{equation}
both given by
$$
\mat{a}{b}{c}{d}\mapsto \frac{\Z^2}{\Z(a,b)+\Z(c,d)}.
$$
If $d$ is square-free then the both side of the second bijection are single points.
\end{prop}
\begin{proof}
 First note that the induced maps are both well-defined. The first bijection is already proved and used in \S\ref{1march2012}.
 The action of $\SL 2\Z$ from the left on  $\Mat_d(2,\Z)$ corresponds to the base change in the lattice 
$\Z(a,b)+\Z(c,d)$ in the right hand side of the bijection. The action of $\SL 2\Z$ from the right on  $\Mat_d(2,\Z)$ corresponds to the isomorphism of finite
groups  in the right hand side of the bijection.
\end{proof}
Any element in $S_d$ is isomorphic to the group $\frac{\Z}{d_2\Z}\times \frac{\Z}{d_1d_2\Z}$ for some $d_1,d_2\in\N$ with $d=d_2^2d_1$.
In the right hand side of (\ref{clarice}) the corresponding element is represented by the matrix $\mat{d_1d_2}{0}{0}{d_2}$.
In the geometric context, this means that any isogeny of elliptic curves $E_1\to E_2$ over an algebraically closed field 
can be decomposed into $E_1\stackrel{\alpha}{\to} E_1\stackrel{\beta}{\to} E_2$, where $\alpha$ is the multiplication by $d_2$ 
and $\beta$ is a degree $d_1$ isogeny with cyclic center. Note that
$$
\sigma(d):=\sum_{d=d_2^2d_1}\psi(d_1).
$$
We conclude that we have a natural decomposition of the both geometric and holomorphic Hecke operators:
$$
T_dt=\sum_{d=d_2^2d_1} d_2^{-m-2}\cdot T_{d_1}^0t, \ \ \ t\in \mf{n}{m}
$$
where in the geometric context
$$
T_d^0: \mf{n}{m}\to \mf{n}{m},\  T_d^0(t)(E,\omega)=\frac{1}{d}\sum_{f:E_1\to E,\  \deg(f)=d, \ker(f)\hbox{ is cyclic }}  t(E_1, f^*\omega),
$$
and in the holomorphic context
$$
T_d^0:\mf{n}{m}\to \mf{n}{m},\ \ T_d^0f=\sum_{ A\in (\SL 2\Z\backslash \Mat_d(2,\Z))^0}  f||_mA.
$$
Here, $(\SL 2\Z\backslash \Mat_d(2,\Z))^0$ is the fiber of the map 
$$
\SL 2\Z\backslash \Mat_d(2,\Z)\to \SL 2\Z\backslash \Mat_d(2,\Z)/\SL 2\Z 
$$
over the matrix $\mat{d}{0}{0}{1}$. 
Note that in the geometric context using the Galois action $\Gal (\bar\k/\k)$, we can see that
 $T_d^0$ is defined over the field $\k$ and not its algebraic closure. 
We call $T_d^0$ the refined Hecke operator.  The refined Hecke operator $T_d^0$ will be used in the next sections. Note that if $d$ is 
square free then $T^0_d=T_d$.


\section{Proof of Theorem \ref{7nov2011} }
\label{27nov2011-4}
For a holomorphic quasi-modular form of weight $m$ we associate the polynomial
$$
P_f^0(x):=\prod_{A\in (\SL 2\Z\backslash \Mat_n(2,\Z))^0 }(x-d\cdot f||_{m}A)=\sum_{j=0}^{\psi(d)} P_{f,j}^0 x^j.
$$
\begin{prop}
$P_{f,j}^0$ is a full quasi-modular form of weight $(\psi(d)-j)\cdot m$.
\end{prop}
\begin{proof}
 The coefficient  $P_{f,j}^0$ of $x^j$   is 
a homogeneous polynomial with rational coefficients and of degree $\psi(d)-j$  in 
$$
T_d^0(f^k),\ k=1,2,\ldots,\psi(d)-j,\ {\rm weight}(T_d^0(f^k))=k,
$$
where $T_d^0:\mf{}{}\to \mf{}{} $ is the refined $d$-th Hecke operator defined in \S \ref{payebabanemikham}. 
For instance, the coefficient of 
$x^{ \psi(d)-1}$ is 
$-d\cdot T_d^0f$ and the coefficient of $x^{ \psi(d)-2}$ is  $\frac{d}{2}((T_d^0f)^2-T_d^0f^ 2)$. 
Now the assertion follows from the fact that the Hecke operator $T_d^0$
sends a quasi-modular form of weight $m$ to a quasi-modular form of weight $m$.  
\end{proof}

Using the fact that the algebra of quasi-modular forms over $\Q$ is isomorphic to $\Q[E_2,E_4,E_6]$ we conclude that 
$P_f^0(x)$ is a homogeneous polynomial  of degree $\psi(d)\cdot m$ in the ring
$$
\Q[x,E_2,E_4,E_6],\ {\rm weight} (E_{2k})=2k,\ \ k=1,2,3,  \ \ {\rm weight}(x)=m.
$$
The geometric definition of the polynomial $P_f^0(x)$ is:
\begin{equation}
\label{15mar2012}
P_f^0(x)(E,\omega)=\prod (x-f(E_1,g^*\omega))
\end{equation}
where the product is taken over all degree $d$ isogenies $g: E_1\to E$ with cyclic kernel.

\begin{proof}[Proof of Theorem \ref{7nov2011}] 
Let us regard $s_i$'s as holomorphic functions on the upper half plane and $t_i$'s as variables. We define
$$
I_{d,i}=P_{s_i}^0(t_i),\ \ i=1,2,3.
$$
For $i=2,3$,  $T_d^0 s_i^k$ is a homogeneous polynomial of degree $ki$ in $\Q[s_2,s_3],\ {\rm weight}(s_i)=2i$ and for $i=1$, $T_d^0s_i^k$ 
is a homogeneous polynomial of degree $ki$ in $\Q[s_1,s_2,s_3],\ {\rm weight}(s_i)=2i$.
We conclude that  $I_{d,i}(t_i,s_1,s_2,s_3)$ is a homogeneous polynomial of degree
$i\cdot \psi(d)$ in the weighted ring (\ref{28mar2012}) and for $i=2,3$ it does not depend on $t_1$. 
By definition it is monic in $t_i$. We have
$$
s_i||_{2i}\mat{d}{0}{0}{1}=d^{2i-1}t_i(d\tau)
$$
and so $I_{d,i}(d^{2i}\cdot t_i(d\cdot \tau), s_1(\tau),s_2(\tau),\ s_3(\tau))=0$.
\end{proof}

\section{Proof of Theorem \ref{23nov2011}} 
\label{27nov2011-5}
Since the period map $\per$ is a biholomorphism, it is enough to prove the same statement for 
the push-forward of $\Ra$ and $V_d$ under the product of two period maps: 
$$
\per\times \per: T\times T\to \SL 2\Z\backslash \pedo\times \SL 2\Z\backslash \pedo.
$$
First we describe the push-forward of $V_d$. Using Theorem \ref{theo-1} and the comparison of Hecke operators in 
both algebraic and complex context, we have:
$$
V_d\cap (T\times T)=
$$
$$
\{((E_1,\omega_1), (E_2, \omega_2))\in T\times T\ \mid\  \exists f:E_1\to E_2,\ f^*\omega_2=\omega_1, \ \ker(f) \hbox{ is cyclic of order } d \}.
$$
Let us now consider the elliptic curves $E_i,\ i=1,2$ as complex curves.
For a $d$-isogeny $f:E_1\to E_2$ such that $\ker(f)$ is cyclic, we can take a symplectic basis $\delta_1,\gamma_1$ of $H_1(E_1,\Z)$ and  $\delta_2,\gamma_2$ of $H_1(E_2,\Z)$
such that 
$$
f_*\delta_1=d\cdot \delta_2,\ \  f_*\gamma_1= \gamma_2.
$$
Therefore, if the period matrix associated to $(E_i,\{\alpha_i, \omega_i\},\{\delta_i\,\gamma_i\}), \ i=1,2$ 
is denoted respectively by $x$ and $y$ then 
$$ 
x=\pi_d(y):=\mat{y_1}{dy_2}{d^{-1}y_3}{y_4}.
$$
Therefore, the push-forward  of $V_d$ under $\per\times \per$ and then its pull-back to 
$\pedo\times\pedo$ is given by:
$$
V_d^*=\{(\pi_d(y),y)\mid y\in \pedo\}.
$$
The push-forward of the vector field $\Ra_d$ by $\per\times \per$ and then its pull-back in $\pedo\times\pedo$ 
is given by the vector field
$$
\Ra^*=d (y_2\frac{\partial}{\partial y_1}+y_4\frac{\partial}{\partial y_3})-(x_2\frac{\partial}{\partial x_1}+x_4\frac{\partial}{\partial x_3})
$$
where we have used the coordinates $(x,y)$ for $\pedo\times\pedo$.
Now, it can be easily shown  that the above vector field $\Ra^*$  is tangent to $V_d^ *$. 

\begin{rem}\rm
The locus $V_d^*$
contains the one dimensional locus:
\begin{equation}
\label{50copos}
\tilde \uhp:= \{\left(\mat{\tau}{-d}{d^{-1}}{0}, \ \mat{\tau}{-1}{1}{0}\right)\mid \tau\in\uhp\}.
\end{equation}
Note also that the push-forward of the Ramanujan vector field $\Ra$  is tangent to the image of $\uhp\to \pedo$ and restricted to
this locus it  is $\frac{\partial}{\partial \tau}$. Therefore,  $\Ra^*$
is tangent to the locus $\tilde \uhp$ and restricted to there is again 
$\frac{\partial}{\partial \tau}$.

\end{rem}


\section{Proof of Theorem \ref{07nov2011}}
\label{27nov2011-6}
From Theorem \ref{7nov2011} it follows that  $I_{d,1}$ is a linear
combination of the monomials (\ref{8july2011}) and $I_{d,i},\ i=2,3$ is a linear 
combination of the same monomials with $a_1=0$.  The proof is  a slight modification of an argument in holomorphic foliations, see \cite{pereira}.
We prove that $J_{d,i}$'s restricted to $V_d$ are identically zero. We know that $I_{d,i}$ is a linear combination of $\alpha_j$'s with $\C$ 
(in fact $\Q$) coefficients: 
$$
I_{d,i}=\sum c_j\alpha_j.
$$
Since $\Ra_d$ is tangent to the variety $V_d$, we conclude that $\Ra_d^r(\sum c_j\alpha_j)=\sum c_j\Ra^r_d\alpha_j$ restricted to
$V_d$ is zero. This in turn implies that the matrix used in the definition of $J_{d,i}$ restricted  to $V_d$ 
has non-zero kernel and so its determinant restricted to $V_d$ is zero.


\section{Examples and final remarks}
\label{27nov2011-7}
In order to calculate $I_{d,i}$'s using Theorem \ref{07nov2011} we can proceed as follow: 
We use the Gr\"obner basis algorithm and find the 
irreducible components of the affine variety given by the ideal  $\langle J_{d,1},J_{d,2},J_{d,3}\rangle$ and among them identify the variety $V_d$. 
In practice this algorithm fails even for the simplest case $d=2$. In this case 
we have $\deg(J_{2,1})=42,\ \deg(J_{2,2})=40,\ \deg(J_{2,3})=69$ and calculating 
 the Gr\"obner basis of  the ideal $\langle J_{d,1}, J_{d,2},J_{d,3}\rangle$ is a huge amount
of computations. We use the $q$-expansion of $t_i$'s and we calculate $I_{d,i}, i=1,2,3,\ d=2,3$. 
We have written powers of $t_i$ in the first row and the corresponding coefficients  in the second row. For more examples 
see the author's web-page.  
$$
I_{2,1}: \left(
\begin{array}{*{4}{c}}
t_{1}^{3} & t_{1}^{2} & t_{1} & 1 \\
1 & -6s_{1} & 12s_{1}^{2}-3s_{2} & -8s_{1}^{3}+6s_{1}s_{2}-2s_{3}
\end{array}
\right)
$$
$$
I_{2,2}:\left(
\begin{array}{*{4}{c}}
t_{2}^{3} & t_{2}^{2} & t_{2} & 1 \\
1 & -18s_{2} & 33s_{2}^{2} & 484s_{2}^{3}-500s_{3}^{2}
\end{array}
\right)
$$

$$
I_{2,3}:\left(
\begin{array}{*{4}{c}}
t_{3}^{3} & t_{3}^{2} & t_{3} & 1 \\
1 & -66s_{3} & -1323s_{2}^{3}+1452s_{3}^{2} & 10584s_{2}^{3}s_{3}-10648s_{3}^{3}
\end{array}
\right)
$$
$$
I_{3,1}:\left(
\begin{array}{*{5}{c}}
t_{1}^{4} & t_{1}^{3} & t_{1}^{2} & t_{1} & 1 \\
1 & -12s_{1} & 54s_{1}^{2}-24s_{2} & -108s_{1}^{3}+144s_{1}s_{2}-64s_{3} & 81s_{1}^{4}-216s_{1}^{2}s_{2}-48s_{2}^{2}+192s_{1}s_{3}
\end{array}
\right)
$$
$$
I_{3,2}:\left(
\begin{array}{*{5}{c}}
t_{2}^{4} & t_{2}^{3} & t_{2}^{2} & t_{2} & 1 \\
1 & -84s_{2} & 246s_{2}^{2} & 63756s_{2}^{3}-64000s_{3}^{2} & 576081s_{2}^{4}-576000s_{2}s_{3}^{2}
\end{array}
\right)
$$

{\tiny
$$
I_{3,3}:
\left(
\begin{array}{*{5}{c}}
t_{3}^{4} & t_{3}^{3} & t_{3}^{2} & t_{3} & 1 \\
1 & -732s_{3} & -169344s_{2}^{3}+171534s_{3}^{2} & 11007360s_{2}^{3}s_{3}-11009548s_{3}^{3} & -502020288s_{2}^{6}+939266496s_{2}^{3}s_{3}^{2}-437245479s_{3}^{4}
\end{array}
\right)
$$
}
The polynomial $I_{d,i},\ i=2,3$ is monic in $t_i$ and it has integral coefficients. This follows from the formula (\ref{15mar2012}) and the fact that 
geometric modular forms can be defined over a field of characteristic $p$. Our computations shows that $I_{d,1}$ is also monic in $t_1$ and it is with 
integral coefficients. However, the notion of a geometric quasi-modular form is only elaborated over field of characteristic zero (see \cite{ho14}). 
The proof of such a statement for $I_{d,1}$ would need a reformulation of 
the definition of geometric quasi-modular forms.

For $j=0,1,\ldots,\psi(d)$, the coefficients of  $t_i^j$ in  $I_{d,i},\ i=2,3$ is homogeneous of degree $i\cdot \psi(d)-i\cdot j$ in $s_2$ and $s_3$
with ${\rm weight}(s_2)=2$ and ${\rm weight}(s_3)=3$. From this it follows that $I_{d,i}$ can be written in a unique way as a polynomial
in $t_i,s_i$ and $s_2^3-s_3^2$, say it $I_{d,i}=P_{d,i}(t_i,s_i, s_2^3-s_3^2)$. It follows that 
$$
\langle I_{d,1}, I_{d,2}, I_{d,3},  s_2^3-s_3^2\rangle= \langle I_{d,1}, P_{d,2}(t_2,s_2,0), P_{d,3}(t_2,s_2,0), s_2^3-s_3^2\rangle
$$
and so 
the irreducible components of
the variety $I_{d,1}=I_{d,2}=I_{d,3}=s_2^3-s_3^2=0$ are given by 
\begin{equation}
 \label{jashn}
t_2-a_2s_2=t_3-a_3s_3=s_2^3-s_3^2=I_{d,1}=0,
\end{equation}
where $a_i,\ i=2,3$ is a root of the polynomial $P_{d,i}(t_i,1,0)$. We have also
$$
\langle I_{d,1}, I_{d,2}, I_{d,3},  t_2^3-t_3^2\rangle=\langle I_{d,1}, I_{d,2}, I_{d,3},  t_2^3-t_3^2, B(s_2,s_3)\rangle,
$$
where $B(s_2,s_3)$ is the resultant of the polynomials $I_{d,2}(t^3,s_2,s_3)$ and $I_{d,3}(t^2,s_2,s_3)$ with respect to the variable $t$. From the definition
of the resultant it follows that $B$ is in the ideal generated by $I_{d,i},\ i=2,3$ and $t_2^3-t_3^2$. Finally, 
we conclude that the variety $V_d\cap\{ s_2^3-s_3^2=0\}$ has many irreducible components of dimension $2$ and given by (\ref{jashn}) and 
$V_d\cap\{ t_2^3-t_3^2=0\}$ is an irreducible of dimension $2$. By the arguments in \S\ref{27nov2011-5} we know that 
$V_d\cap (T\times T)$ is irreducible and so $V_d$ is also irreducible.



\def\cprime{$'$} \def\cprime{$'$} \def\cprime{$'$}

\end{document}